\documentclass[11pt, letterpaper]{article}
\usepackage[margin=2.5cm]{geometry}
\usepackage{helvet}
\usepackage{courier}
\usepackage{url}
\usepackage{bm}
\usepackage{graphicx}
\usepackage{tcolorbox}
\usepackage[dvipsnames]{xcolor}
\usepackage{enumerate}
\usepackage{pgfplots}
\usepackage{algorithm}
\usepackage{eucal}
\usepackage{mathtools}
\usepackage{float}
\usepackage{amsthm}
\usepackage{amsmath,amssymb,amsfonts}
\usepackage{xcolor}
\usepackage{bbold}
\usepackage{xspace}
\usepackage{lettrine}
\usepackage{subfigure}
\usepackage{tikz}
\usetikzlibrary{backgrounds,automata}
\usetikzlibrary{shapes.geometric}
\usepackage{latexsym}
\usepackage{fancybox}
\theoremstyle{plain}
\usetikzlibrary{backgrounds,automata}
\usetikzlibrary{shapes.geometric}
\usepackage{latexsym}
\usepackage{fancybox}
\usepackage{hyperref}
\hypersetup{
    colorlinks=true,    
    citecolor=blue,     
    linkcolor=blue,     
    urlcolor=blue       
}

\theoremstyle{plain}

\newtheorem{theorem}{Theorem}[section]
\newtheorem{definition}[theorem]{Definition}
\newtheorem{lemma}[theorem]{Lemma}
\newtheorem{claim}[theorem]{Claim}

\newtheorem{corollary}[theorem]{Corollary}

\newtheorem{observation}[theorem]{Observation}
\newcommand{\Xsum}{|X_1|+|X_2|+\ldots+|X_n|}
\newcommand{\wXsum}{|X_1|+2|X_2|+\ldots +n|X_n|}

\newcommand{\tXsum}{|X_2|+|X_3|+\ldots+|X_n|}

\usepackage{setspace}
\usepackage{amsmath} 
\usepackage{cleveref}
\usepackage{graphicx}
\usepackage{authblk}

\Crefname{theorem}{Theorem}{Theorems}
\Crefname{claim}{Claim}{Claims}
\Crefname{observation}{Observation}{Observations}
\usepackage[backend=biber]{biblatex}
\usepackage{algpseudocode}
\begin{document}
\title{An upper bound on the proper hat guessing number of graphs}
\author{Ioannis Kakatelis \thanks{Czech Technical University,Czech Republic. \emph{E-mail}: \text{ioannis.kakatelis@fit.cvut.cz}.}
}

\date{}

\maketitle

\begin{abstract}
We study the proper hat guessing game on graphs, introduced by Adriaensen et al. \cite{adriaensen2026hatguessingpropercolorings}. In this game, the players are seated on the vertices of a graph $G$ and assigned hats from a set of $q$ colors such that the resulting assignment forms a proper coloring. The visibility of each vertex is limited to the hat colors of their neighborhood. Then they must simultaneously output a guess about the color of their own hat. The players win if at least one guess is correct. A parameter related to this problem is the proper hat guessing number $\operatorname{HG}_{P}(G)$ that is the maximum number of colors $m$ such that the players can guarantee a winning strategy. Motivated by the work of Shurman et al. in the upper bound on the ordinary hat guessing number, we establish the first upper bound that depends both on the number of vertices $n$ and the maximum degree $\Delta$ in the case where $\Delta \geq \frac{n}{e+1}$. This result leads us to show that the proper hat guessing number of the binomial random graph $G_{n,1/2}$ is bounded above by $cn$, where $c \approx 1.366$. Finally, we prove that graphs of maximum degree $(1-\gamma)n$ for some fixed $\gamma \in (0,1]$ cannot have $\operatorname{HG}_{P}(G) = (2-o(1))n$.
\end{abstract}

\begin{section}{Introduction}
In recent years, hat guessing games have attracted significant attention from a combinatorial perspective. The game is played on a graph $G = (V,E)$, where the vertices represent the players and each player receives a hat whose color is drawn from a fixed set $[q] = \{1,2,\ldots,q\}$. Each player has restricted visibility and can observe only the hat colors of their neighbors in $G$. Prior to the adversary assigning the hats, the players agree upon a deterministic guessing strategy, after which no further communication is permitted. The players must then simultaneously output a guess for the color of their own hat. Their objective is to find a strategy that guarantees at least one correct guess for every valid hat assignment chosen by the adversary. The parameter associated with this game is the hat guessing number, denoted by $\operatorname{HG}(G)$, which is defined as the largest integer $q$ for which the players can guarantee a winning strategy.

The classical hat guessing number has been studied extensively in a series of papers during last years \cite{alon2020hat,alon2022hat,wang2023hat,he2020hat,knierim2023hat,latyshev2021hats,latyshev2024hat}. Recently, other variants of this game has been considered allowing the hat guessing number to be fractional \cite{blavzej2023bears} or allowing the players to have more than one guesses \cite{bosek2021hat}. In this paper we will be focused on a variant that was introduced by Adriaensen et al.~\cite{adriaensen2026hatguessingpropercolorings} where now the adversary is placing the hats in a way that produces a proper coloring in the graph, such that no two neighboring players will receive the same hat color. The parameter of interest in this variant is denoted by $\operatorname{HG}_{P}(G)$ as the largest number $q$ of colors such that for graph $G$ there exists a strategy where for every proper coloring of G using $q$ colors, at least one vertex guesses correctly.
\newpage
The question posed by Alon et al. \cite{alon2020hat} of whether the hat-guessing number admits an upper bound depending solely on the maximum degree has attracted considerable attention within the research community.
 An evidence on that question was provided in \cite{alon2020hat}, that for every graph $G$ with maximum degree $\Delta$ via an application of Lov\'{a}sz Local Lemma gives that $\operatorname{HG}(G) < e\Delta$. A very recent result in this direction was proved by \cite{shurman2026upper} where they found an upper bound of $\operatorname{HG}(G)$ that is bounded by a function that depends from both $n$ and $\Delta$, where $\Delta > \frac{n}{e}$.

This question was also studied in other variants of this game. Regarding the variant of fractional hat guessing number it was shown in \cite{blavzej2023bears} that the fractional hat guessing number of any graph is at least $\Omega\left(\frac{\Delta}{\log\Delta}\right)$. In the proper variant of the game was shown in  \cite{adriaensen2026hatguessingpropercolorings} that $\operatorname{HG}_{P}(G) < (e+1)\Delta$.

\vspace{1em}

Our main result is to establish the first non-trivial upper bound for the proper hat guessing number in terms of the number of vertices \(n\) and the maximum degree $\Delta$.

\begin{theorem}\label{thm:main_bound}
For any graph $G$ on $n$ vertices with maximum degree $\Delta$, 
$$\operatorname{HG}_{P}(G) \leq \frac{3}{2}+ \Delta+ \frac{\sqrt{2n^{2}+2n\Delta-10n+9}}{2}. $$
\end{theorem}
An important application of this bound is to capture the behavior in the proper hat guessing number of the binomial random graph. It is worth mentioning at that point that the maximum degree of the binomial random graph $G_{n,\frac{1}{2}}$ is $(1+o(1))\frac{n}{2} \geq \frac{n}{e} > \frac{n}{e+1}$, so that would imply that the application of Lov\'{a}sz Local Lemma does not yield a non-trivial result in the proper case.

The first result on this direction was shown by Bosek et al. \cite{bosek2021hat} about the ordinary hat guessing number where they derived a lower and an upper bound on the binomial random graph with high probability, when the probability is tending to one and $n$ tends to infinity, 
$$(2-o(1))\log_{2}n \leq \operatorname{HG}(G_{n,1/2}) \leq n-(1+o(1))\log_{2}n.$$
In a following work, Alon and Chizewer \cite{alon2022hat} managed to improve the lower bound and they proved that $\operatorname{HG}(G_{n,1/2}) \geq n^{1-o(1)}$ using the method of identifying a subgraph with arbitrary large ordinary hat guessing number. Later, Wang and Chen \cite{wang2023hat} extended the previously known results by showing that for any constant number $p$, it holds
$$n^{1-o(1)} \leq \operatorname{HG}(G) \leq (1-o(1))n$$ with high probability when $G \in \mathcal{G}(n,p)$ where $\mathcal{G}(n,p)$  denote the Erd\H{o}s--R\'enyi random graphs with n vertices and edge-chosen
probability $p \in (0, 1)$.
Recently, Shurman et al. \cite{shurman2026upper} proved a better upper bound stated as follows, 
$$\operatorname{HG}(G_{n,1/2}) \leq (1+o(1))\left(\frac{1+\sqrt{5}}{4}\right)n \approx 0.809n.$$
As a consequence of our main theorem, we obtain the first upper bound for the proper hat guessing number of the binomial random graph, as stated below.
\vspace{1.5em}
\begin{corollary}\label{cor:G_n,1/2}
With high probability, 
$$\operatorname{HG}_{P}(G_{n,1/2}) \leq (1+o(1)) \left(\frac{1+\sqrt 3}{2}\right)n \approx 1.366n$$
\end{corollary}
\begin{corollary}\label{cor:epsilon_delta}
For any graph $G$ on $n$ vertices with maximum degree $\Delta$ we have that: 
$$\operatorname{HG}_{P}(G) \leq \frac{5\Delta}{4} + \frac{3n}{4} + \frac{3}{2}.$$
\end{corollary}
According to the previous corollary we have that for any $\gamma > 0$ and any graph $G$ on $n$ vertices, for $n$ to be sufficiently large if $\Delta(G) \leq (1-\gamma)n$ then $\operatorname{HG}_{P}(G) \leq (2-\zeta)n$ for some $\zeta(\gamma)$ independent of $n$. \\ 

\textbf{Main proof outline.} Let $G$ be a graph where each vertex is given one out  of $k$ available colors, but in the restriction that neighboring vertices receive different colors and let $S$ be a winning strategy for the proper coloring game on the graph $G$. If a coloring $c$ was chosen at random, then if we condition on the colors of the neighbors, we can see that each vertex $v$ guesses their hat color correctly with probability $ \leq \frac{1}{k-\Delta}$, where $\Delta \coloneqq \Delta(G)$ to be the maximum degree in the graph $G$. We define as well $R_{S}(c)$ to be the number of correct guesses under the proper coloring $c$ and the strategy $S$.
\begin{lemma}\label{lem:exp_Delta}
$\mathbb{E}[R_S(c)] \leq \frac{n}{k-\Delta}$
\end{lemma}
\begin{proof}
Let us define the event \[
S_u = \{\, \text{vertex}\  u \text{ guesses correctly under coloring } c \,\}.
\] Now we know that the number of available colors are $k - |c(N(u))|$ and $|c(N(u))| \leq \Delta$, so there are at least $k-\Delta$ valid available choices for vertex $u$. So the probability of the event $S_{u}$ is now $\mathbb{P}(S_{u}) \leq \frac{1}{k-\Delta}$ and by applying the linearity of expectation we can get that 
$$\mathbb{E}[R_{S}(c)] = \sum_{u \in V} \mathbb{P}(S_{u}) \leq \frac{n}{k-\Delta}$$.
\end{proof}
\end{section}
\begin{section}{Construction of the bipartite graph $H$}
The construction is heavily inspired by the one that was provided in \cite{shurman2026upper} adapting it to the case of proper colorings. Here we note that $X_{1}$ represents the set of all proper colorings of $G$ where exactly one vertex guesses correctly its hat color. The rest of the construction is based on the idea to find the underlying colorings that lie in $X_{1}$ and their nearby colorings, these coloring will play a key role to obtain the upper bound that will lead us to our desired result.
\begin{definition}
Let $G$ be a graph on $n$ vertices where each vertex is assigned one of $k$ colors, and let $S$ be a winning guessing strategy on $G$. Define the vertex set of the auxiliary graph $H$ as the set of all colorings using $k$ colors of $G$. We call this set $C$. In $H$, colorings $c \in X_1$ and $c' \notin X_1$ are adjacent if for some $u \in V(G)$:
\begin{enumerate}[(P\arabic*)]
    \item If $v$ is the unique vertex that guessed correctly in $c$, $u \neq v$ and $uv \notin E(G)$. 
    \item For all $w \in V(G) \setminus \{u\}$, $c(w)=c'(w)$.
    \item $S$ guesses the color of $u$ correctly in $c'$. 
\end{enumerate}
\end{definition}

To properly analyze this auxiliary graph, we will need to more precisely describe the colorings not in $X_1$. For each $2 \leq i \leq n$, define:
$$X_i:= \left\{c \in C:R_{S}(c)=i, \deg_H(c) \neq 0  \right\}$$ 
where $R_{S}(c)$ is the number of vertices that guess correctly when $G$ is colored by $c$. We exclude the colorings with degree $0$ in $H$ because they cannot be analyzed with our approach.
Let $Y$ be the set of these excluded colorings:
$$Y:=C \setminus \bigcup_{i=1}^n X_i.$$
Note that the $X_i$ for ${1 \leq i \leq n}$, together with $Y$, partition the vertices and the set of proper colorings in $H$, so: $$\sum_{i=1}^n |X_i|+|Y|=|C|= P(G,k).$$
where $P(G,k)$ represents the chromatic polynomial of the graph $G$ and denotes the number of proper $k-$colorings of the graph $G$. \\ 
The previous definitions and the rule of properness leads us to the following important inequality:
\begin{lemma}\label{lem:eq_first}
If $\frac{n}{k-\Delta} < 2$, $$\frac{n}{k-\Delta} \geq \frac{|X_{1}| + 2|X_{2}|+\ldots+n|X_{n}|}{X_{1}| + |X_{2}| + \ldots+|X_{n}|}$$
\end{lemma}
\begin{proof}
From \cref{lem:exp_Delta} we have that $\mathbb{E}[R_{S}(c)] \leq \frac{n}{k-\Delta}$. Let $b = \mathbb{E}_{c \in B}[R_{S}(c)]$ where $B$ is the set of excluded colorings and denote that $b\geq 2$ since according to the construction each coloring that does not being to $X_{1}$ guesses correctly on at least $2$ vertices. so finally we get: 
\begin{align*}
             \frac{n}{k-\Delta} &\geq \mathbb{E}[R_{S}(c)] \\
             &=\frac{1}{|C|} \sum_{c\in  C}R_{S}(c) \\
            &=\frac{ \wXsum +y\,|Y|}{|X_1|+|X_2|+\ldots +|X_n|+|Y|}  \\
             &\geq \frac{|X_1|+2|X_2|+\ldots+ n|X_n|}{|X_1|+|X_2|+\ldots+ |X_n|}.
        \end{align*}
        Here the last inequality holds because $y\geq 2$ and we assumed $\frac{n}{k-\Delta}< 2$. Note that if $Y$ is empty, the third and fourth lines are equivalent since $|Y|=0$.

\end{proof}
\end{section}
\begin{section}{Analysis of the bipartite graph $H$}
In this section we want to lower bound $\frac{n}{k-\Delta}$, by double counting the number of edges $e(H)$ in the auxillary bipartite graph $H$. First we note that $H$ is a bipartite graph with parts $X_{1}$ and $C \setminus X_{1}$ and since we allow edges only between $c \in X_{1}$ and $c' \notin X_{1}$ we have the following observation. 
\vspace{3em}
\begin{observation}\label{obs:main_equal_double}
$$e(H) = \sum_{c \in X_{1}} \deg(c) = \sum_{c' \notin X_{1}} \deg(c')$$
\end{observation}
Using the \Cref{obs:main_equal_double} the goal is to prove the following bounds: 
\begin{lemma}\label{lem:upper_lower}
$$(n-1-\Delta)|X_{1}| \leq e(H) \leq \sum_{i=2}^{n}i(k-2)|X_{i}|$$
\end{lemma}
\begin{proof}
Let the colorings $c \in X_{1}$ and $c' \notin X_{1}$. The goal is to prove both upper and lower bound by using the \Cref{obs:main_equal_double} and finding a way to bound the degrees under colorings $c$ and $c'$.
\end{proof} 
The proof of \Cref{lem:upper_lower} will be explained in the two following claims.
\begin{claim}\label{clm:clm_upper}
For any proper coloring $c' \in X_{i}, i \neq 1$, 
$$\deg(c') \leq i(k-2)$$
\end{claim}

\begin{proof}
Let $c' \in X_{i}, i \neq 1$ and let $c \in X_{1}$ to be adjacent to $c'$. According to the definition of $H$ the only difference between $c'$ and $c$ is the color of vertex $u$. Under coloring $c'$ vertex $u$ guesses its color correctly and under coloring $c$, $u$ must guess incorrectly because the single correct guess under $c$ belongs to some other vertex. Because no color was modified, the neighborhood of $u$ remains unchanged. If we keep the same guess then the coloring would remain in $X_{i}$ and we want to push the state of the vertex from $X_{i}$ to $X_{1}$. We want to ensure that $u$ becomes incorrect and to do that we must subtract the color that is assigned to $u$ under the coloring $c'$.

We know that an edge in $H$ connects a coloring under $c' \in X_{i}$ that $i$ players guessed correctly with a coloring $c \in X_{1}$, where exactly $1$ player guesses correctly. Furthermore, because $c$ must be a valid proper coloring, the new color assigned to $u$ cannot match any color in its neighborhood $c'(N(u))$. Because $c'$ is a valid proper coloring, $u$'s correct color is disjoint from its neighborhood. Therefore, we must subtract $\vert{}c'(N(u))\vert{}$ proper colors.
So the valid ways for each of the vertices to change their color is $k-2$. To derive our bound because $c$ is adjacent to $c'$,  $\deg(c') \leq \sum_{i \in X_{i}}(k-2) = i(k-2)$ so there are at most $i(k-2)$ colorings $c$ adjacent to $c'$.
\end{proof}
\begin{claim}\label{clm:2.4clm} 
For any $c\in X_1$, 
    $$\deg (c)\geq n-1-\Delta.$$   
\end{claim}
\begin{proof}
    Let $c \in X_1$, and let $v \in V(G)$ be the vertex that guesses its color correctly in $c$. Let $u \in V(G)\setminus \{v\}$ be a vertex not adjacent to $v$, and let $a_u$ be the color that $u$ guesses under the coloring $c$. 
        Let $c_u \in C$ be the coloring where:
        $$c_u(x):= \begin{cases} a_u \text{ if }x = u \\c(x) \text{ otherwise.}
        \end{cases}$$
        Then $c$ and $c_u$ are adjacent in $H$. For each $u \in V(G)$ not adjacent to $v$, the colorings $c_u$ are distinct, and there are at least $n-1-\Delta$ vertices not adjacent to $v$, not including $v$ itself. Thus, $\deg(c)\geq n-1-\Delta$. 
\end{proof}
\vspace{5em}
Using now both \Cref{clm:2.4clm,clm:clm_upper} in combination with the \cref{obs:main_equal_double} we get the desired result as follows 

$$e(H)=\sum_{c\in X_1}  \deg (c)\geq (n-1-\Delta)|X_1|,$$ and $$e(H)=\sum_{c' \notin X_1} \deg(c')\leq \sum _{c' \notin X_1, c' \in X_i} i(k-2)=\sum_{i=2}^{n} i(k-2)|X_i|.$$
This completes the proof of \Cref{lem:upper_lower}.

The goal now is to find a way to lower bound $\frac{n}{k}$ but to do that we need first to bound the fraction: $$\frac{\wXsum}{\Xsum}$$
\begin{lemma}\label{lem:sum_lower_bound}
$$\wXsum \geq\Xsum + \frac{1}{2}|X_1|\frac{n-1-\Delta}{k-2}$$
\end{lemma}
\begin{proof}
Dividing the inequality from \Cref{lem:upper_lower} by $k-2$, we obtain
\begin{equation}\label{eq:3.2dividedbyk}
     2|X_2|+3|X_3|+4|X_4|+\ldots +n|X_n| \geq \frac{|X_1|(n-1-\Delta) }{k-2}.
\end{equation}
Thus, 
\begin{align*}
    \wXsum
    &= \Xsum + |X_2| + 2|X_3| + 3|X_4| + \cdots + (n-1)|X_n| \\
    &= \Xsum
       + \frac{1}{2}\bigl(
           2|X_2| + 4|X_3| + 6|X_4| + \cdots + 2(n-1)|X_n|
         \bigr) \\
    &\geq \Xsum
       + \frac{1}{2}\bigl(
           2|X_2| + 3|X_3| + 4|X_4| + \cdots + n|X_n|
         \bigr) \\
    &\geq \Xsum
       + \frac{1}{2}|X_1|
         \frac{n-1-\Delta}{k-2},
       \qquad\text{by \eqref{eq:3.2dividedbyk}.}
\end{align*}
\end{proof}
Now we bound the denominator $\Xsum$.
\begin{lemma}\label{3.6:lem-den} 
If $\frac n {k-\Delta} < 2$, then
        $$\Xsum \leq \frac{|X_1|}{2-\frac{n}{k-\Delta}}.$$
    \end{lemma}
    \begin{proof}
        Let $$a=\frac{|X_1|}{\Xsum},b=\frac{\tXsum}{\Xsum}.$$
       These give us the linear system: 
       \begin{align*}
           &a+b=\frac{\Xsum}{\Xsum}=1 \\
           &a+2b=\frac{|X_1|+2(|X_2|+|X_3|+\ldots+ |X_n|)}{|X_1|+|X_2|+|X_3|+\ldots+ |X_n|}\leq \frac{|X_1|+2|X_2|+3|X_3|+\ldots+ n|X_n|}{|X_1|+|X_2|+|X_3|+\ldots+ |X_n|}\leq \frac{n}{k-\Delta}.
       \end{align*}
        The last inequality is from Lemma $\ref{lem:eq_first}$, which we can apply since $\frac n {k-\Delta}< 2$.
        Solving the linear system, we obtain that 
        $$a\geq 2-\frac{n}{k-\Delta}.$$

        By definition, $a(\Xsum)=|X_1|$, and since $a\geq 2-\frac{n}{k-\Delta}$, we have $$|X_1|=a(\Xsum)\geq (2-\frac{n}{k-\Delta})(\Xsum), $$ so $$\Xsum \leq \frac{|X_1|}{2-\frac{n}{k-\Delta}}.$$
    \end{proof}

Putting these last two lemmas together, we obtain:
\vspace{3em}

\begin{lemma}\label{3.7lem:bound} 
    $$\frac n {k-\Delta} \geq \min \left(2,\, 1+\left(2-\frac{n}{k-\Delta}\right)\left(\frac{1}{2}\frac{n-1-\Delta}{k-2}\right)\right).$$
\end{lemma} 
\vspace{1em}
\begin{proof}
If $\frac n {k-\Delta} \geq 2,$ then the lemma holds. Thus, we  assume $\frac n {k-\Delta} < 2$, allowing us the use of Lemma \ref{lem:eq_first} and Lemma \ref{3.6:lem-den}.
    \begin{align*}
        \frac{n}{k-\Delta} &\geq \frac{\wXsum}{\Xsum}  & \text{by \ref{lem:eq_first}}
        \\ & = 1+\frac{|X_2|+2|X_3|+\ldots +(n-1)|X_n|}{\Xsum} 
        \\&\geq 1+ \left(2-\frac{n}{k-\Delta}\right) \frac{|X_2|+2|X_3|+\ldots+ (n-1)|X_n|}{|X_1|} &\text{by \ref{3.6:lem-den}}
        \\& \geq 1+\left(2-\frac{n}{k-\Delta}\right)\left(\frac{1}{2}\frac{n-1-\Delta}{k-2}\right). &\text{by \ref{lem:sum_lower_bound}}
    \end{align*}
\end{proof}
\newpage
\begin{proof}[Proof of Theorem \ref{thm:main_bound}]
    Let $G$ be a graph on $n$ vertices where each vertex is colored one of $k$ colors, and let $S$ be a winning strategy on $G$. The goal is to bound $k$ from above.
    
   We start by applying Lemma $\ref{3.7lem:bound}$. 
   $$\frac{n}{k-\Delta}\geq \min \left(2, 1+\left(2-\frac{n}{k-\Delta}\right)\left(\frac{1}{2}\frac{n-1-\Delta}{k-2}\right)\right).$$

   There are two possible cases:

   \medskip

   \noindent\textbf{If $2 \leq 1+\left(2-\frac{n}{k-\Delta}\right)\left(\frac{1}{2}\frac{n-1-\Delta}{k-2}\right)$: }

    By Lemma \ref{3.7lem:bound}, $\frac n {k-\Delta} \geq 2$. Observe that for all $n \geq 2$ and $0 \leq \Delta \leq n-1$, 
$$ k \leq \frac{n}{2} + \Delta \leq  \Delta+\frac{3}{2}+ \frac{\sqrt{2n^{2}+2n\Delta-10n+9}}{2}$$. 
    
   
   
    \noindent\textbf{If $ 2 \geq 1+\left(2-\frac{n}{k-\Delta}\right)\left(\frac{1}{2}\frac{n-1-\Delta}{k-2}\right) $:} 
    
    In this case, Lemma \ref{3.7lem:bound} gives
    $$\frac{n}{k-\Delta}\geq  1+\left(2-\frac{n}{k-\Delta}\right)\left(\frac{1}{2}\frac{n-1-\Delta}{k-2}\right).$$
    Multiplying through by $2(k-2)(k-\Delta)$, $$2n(k-2) \geq 2(k-\Delta)(k-2) + (2(k-\Delta) -n)(n-1-\Delta).$$ 

    Solving the quadratic for $k$, we get
    \begin{align*}
         k &\leq \frac{4\Delta+6+\sqrt{8n^{2}+8n\Delta-40n+36}}{4}
         \\ &= \Delta+\frac{3}{2}+ \frac{\sqrt{2n^{2}+2n\Delta-10n+9}}{2}.
    \end{align*}
    
\medskip
In both  of the case that was explained before, it holds that $k\leq \Delta+\frac{3}{2}+ \frac{\sqrt{2n^{2}+2n\Delta-10n+9}}{2}.$

The equality is attained for all winning strategies $S$ for the proper hat guessing game and we can see that the following holds:
$$\operatorname{HG}_{P}(G) \leq \Delta+\frac{3}{2}+ \frac{\sqrt{2n^{2}+2n\Delta-10n+9}}{2}. $$
We can see in the previous bound that the equality holds only if we replace $G = K_{n}$ and we have  $\operatorname{HG}_{P}(K_n) = 2n-1$ and we know that this is true by \cite{adriaensen2026hatguessingpropercolorings}
\end{proof}
\newpage
\begin{proof}[Proof of Corollary \ref{cor:G_n,1/2}] 
    It is known (see Bollobás, \cite{bollobas2001random}) that the maximum degree of the random graph is $(1+o(1))n/2$. To be more accurate, with high probability, $\Delta \leq \frac n 2 + (1+o(1)) \sqrt{\frac{n\log n}2}$. Let $\Delta = \frac n 2+ \epsilon$, where $\epsilon=o(n)$. We apply our upper bound from \Cref{thm:main_bound} to obtain with high probability,
    \begin{align*}
        \operatorname{HG}_{P}(G) &\leq \Delta + \frac{3}{2} + \frac{\sqrt{2n^2 + 2n\Delta - 10n + 9}}{2}
        \\&= \frac n 2 + \epsilon + \frac 3 2 + \frac 1 2 \sqrt{2n^2 + 2n\left(\frac n 2 + \epsilon\right) - 10n + 9}
        \\&= \frac n 2 + \epsilon + \frac 3 2 + \frac 1 2 \sqrt{2n^2 + n^2 + 2n\epsilon - 10n + 9}
        \\&= \frac n 2 + \epsilon + \frac 3 2 + \frac 1 2 \sqrt{3n^2 + 2n\epsilon - 10n + 9}
        \\&\leq \frac n 2 + o(n) + \frac 1 2 \sqrt{3n^2 + o(n^2)}
        \\&\leq (1+o(1)) \frac{1+\sqrt 3}{2}n.
    \end{align*}

\end{proof}
\begin{proof}[Proof of Corollary \ref{cor:epsilon_delta}]
Let $G$ be a graph on $n$ vertices. By the main bound from \Cref{thm:main_bound} ,
\begin{align*}
\operatorname{HG}_{P}(G) \leq \Delta + \frac{3}{2} + \frac{1}{2}\sqrt{2n^2 + 2n\Delta - 10n + 9}.
\end{align*}

For $n \ge 1$, we note that $-10n+9 \le -1 < 0$, so the expression inside the square root is strictly bounded by $2n^2 + 2n\Delta$. We can factor this as $(2n)(n+\Delta)$ and apply the Arithmetic mean-Geometric mean inequality ($\sqrt{ab} \leq \frac{a+b}{2}$) to obtain:
\begin{align*}
\operatorname{HG}_{P}(G) &\leq \Delta + \frac{3}{2} + \frac{1}{2}\sqrt{(2n)(n+\Delta)} \leq \Delta + \frac{3}{2} + \frac{1}{2}\left(\frac{2n + n + \Delta}{2}\right) = \frac{5\Delta}{4} + \frac{3n}{4} + \frac{3}{2}.
\end{align*}
    To prove the bound when the maximum degree $\Delta$ is less than $(1-\gamma)n$ for some $\gamma \in (0,1]$, observe that
    $$ \operatorname{HG}_{P}(G) \leq \frac{5(1-\gamma)n}{4} + \frac{3n}{4} + \frac{3}{2} = \left(2 - \frac{5\gamma}{4}\right)n + \frac{3}{2}. $$ 
    If we choose $\zeta(\gamma)$ such that $0 < \zeta < \frac{5\gamma}{4}$, then for sufficiently large $n$, $\operatorname{HG}_{P}(G) \leq (2-\zeta)n$.
\end{proof}
\begin{subsection}{AI Disclosure}
I acknowledge the use of Google Gemini for proof checking of \Cref{cor:epsilon_delta} and for grammar checks and text improving in the introduction.
\end{subsection}
\end{section}
\begin{section}{Acknowledgements}
We would like to thank Michal Opler for revising this note and for the fruitful discussions on the broad topic of hat guessing games. This work has been partially supported by the MEYS of the Czech Republic
under the INTER-EXCELLENCE II program (project No. LUAUS26294), the European Union under the project Robotics and advanced industrial production (reg. no. CZ.02.01.01/00/22\_008/0004590) and also supported by the Grant Agency of the Czech Technical University in Prague, grant No.SGS26/181/OHK3/3T/18.
\end{section}
\printbibliography
\end{document}